\documentclass[12pt,reqno]{amsart}
\usepackage{times}

\newtheorem{thm}{Theorem}
\newtheorem{lem}[thm]{Lemma}

\newcommand{\C}{\mathbb{C}}
\newcommand{\D}{\mathbb{D}}

\begin{document}

\title[Logarithmic convexity of area integral means]
{\bf Logarithmic convexity of area integral means\\
 for analytic functions II}
\thanks{Jie Xiao is supported in part by NSERC of Canada and URP of Memorial University}

\author{Chunjie Wang}
\address{Chunjie Wang, Department of Mathematics, Hebei University of Technology,
Tianjin 300401, China} 
\email{wcj@hebut.edu.cn}

\author{Jie Xiao}
\address{Jie Xiao, Department of Mathematics and Statistics, Memorial University
of Newfoundland, St. John's, NL A1C 5S7, Canada}
\email{jxiao@mun.ca}

\author{Kehe Zhu}
\address{Kehe Zhu, Department of Mathematics and Statistics, State University
of New York, Albany, NY 12222, USA} 
\email{kzhu@math.albany.edu}

\begin{abstract}
For $0<p<\infty$ and $-2\le\alpha\le0$ we show that the $L^p$ integral mean on $r\D$ 
of an analytic function in the unit disk $\D$ with respect to the weighted area measure 
$(1-|z|^2)^\alpha\,dA(z)$ is a logarithmically convex function of $r$ on $(0,1)$.
\end{abstract}



\maketitle

\section{Introduction}

Let $\D$ denote the unit disk in the complex plane $\C$ and let
$H(\D)$ denote the space of all analytic functions in $\D$. For
any $f\in H(\D)$ and $0<p<\infty$ the classical integral means of
$f$ are defined by
$$M_p(f,r)=\frac1{2\pi}\int_0^{2\pi}|f(re^{i\theta})|^p\,d\theta,\quad0\le r<1.$$
The well-known Hardy convexity theorem asserts that $M_p(f,r)$, 
as a function of $r$ on $[0,1)$, is non-decreasing and logarithmically convex. 
Recall that the logarithmic convexity of $g(r)$ simply means that $\log g(r)$ is a
convex function $\log r$. The case $p=\infty$ corresponds to the Hadamard Three-Circles 
Theorem. See \cite[Theorem 1.5]{D} for example.

In this paper we will consider integral means of analytic functions in the unit disk 
with respect to weighted area measures. Thus for any real number $\alpha$ we consider the measure
$$dA_\alpha(z)=(1-|z|^2)^\alpha\,dA(z),$$
where $dA$ is area measure on $\D$. For any $f\in H(\D)$ and
$0<p<\infty$ we define
$$M_{p,\alpha}(f,r)=\frac{\displaystyle\int_{r\D}|f(z)|^p\,dA_\alpha(z)}
{\displaystyle\int_{r\D}\,dA_\alpha(z)},\quad 0<r<1,$$ 
and call them area integral means of $f$.

The study of area integral means of analytic functions began in \cite{XZ}, where it was shown
that for $\alpha\le-1$, the function $M_{p,\alpha}(f,r)$ is bounded in $r$ if and only if $f$
belongs to the Hardy space $H^p$, and for $\alpha>-1$, $M_{p,\alpha}(f,r)$ is bounded in $r$
if and only if $f$ belongs to the weighted Bergman space $A^p_\alpha$. See \cite{D} and \cite{HKZ}
for the theories of Hardy and Bergman spaces respectively.

It was also shown in \cite{XZ} that each function $r\mapsto M_{p,\alpha}(f,r)$ is strictly 
increasing unless $f$ is constant. Furthermore, for $p\ge1$ and $\alpha\in\{-1,0\}$, the function
$\log M_{p,\alpha}(f,r)$ is convex in $\log r$. However, an example in \cite{XZ} shows that 
$\log M_{2,1}(z,r)$ is concave in $\log r$. Consequently, the following conjecture was made in
\cite{XZ}: the function $\log M_{p,\alpha}(f,r)$ is convex in $\log r$ when $\alpha\leq0$ and it
is concave in $\log r$ when $\alpha>0$. 

It turned out that the logarithmic convexity of $M_{p,\alpha}(f,r)$ is much more complicated 
than was conjectured in \cite{XZ}. Somewhat surprisingly, the problem is highly nontrivial even
in the Hilbert space case $p=2$. More specifically, it was proved in \cite{WZ} that for $p=2$
and any $f\in H(\D)$ the function $M_{2,\alpha}(f,r)$ is logarithmically convex when $-3\le
\alpha\le0$, and this range for $\alpha$ is best possible. It was also proved in \cite{WZ} that
when $p\not=2$ and $f$ is a monomial, the function $M_{p,\alpha}(z^k,r)$ is logarithmically
convex for $-2\le\alpha\le0$.

Area integral means of analytic functions are also studied in \cite{WX} and \cite{XX}.

The main result of this paper is the following.

\begin{thm}
Suppose $0<p<\infty$, $-2\le\alpha\le0$, and $f$ is analytic in $\D$. Then the function
$M_{p,\alpha}(f,r)$ is logarithmically convex.
\label{1}
\end{thm}

We have been unable to determine whether or not the range $\alpha\in[-2,0]$ 
is best possible. In other words, we do not know if there exists a set $\Omega$ properly 
containing $[-2,0]$ such that $M_{p,\alpha}(f,r)$ is logarithmically convex on $(0,1)$ for 
all $p\in(0,\infty)$, all $\alpha\in\Omega$, and all $f\in H(\D)$. It is certainly 
reasonable to expect that the logarithmic convexity of $M_{p,\alpha}(f,r)$ for all $f$ 
will depend on both $p$ and $\alpha$. The ultimate problem is to find out the precise 
dependence.

\section{Preliminaries}

The proof of Theorem~\ref{1} is ``elementary'' but very laborious. It requires several
preliminary results that we collect in this section. Throughout the paper we use the
symbol $\equiv$ whenever a new notation is being introduced.

The next lemma was stated in \cite{WZ} without proof. We provide a proof here for the sake
of completeness.

\begin{lem}
Suppose $f$ is positive and twice differentiable on $(0,1)$. Then
\begin{enumerate}
\item[(i)] $f(x)$ is convex in $\log x$ if and only if 
$$f'(x)+xf''(x)\ge0$$
for all $x\in(0,1)$.
\item[(ii)] $f(x)$ is convex in $\log x$ if and only if $f(x^2)$ is convex in $\log x$.
\item[(iii)] $\log f(x)$ is convex in $\log x$ if and only if
$$D(f(x))\equiv\frac{f'(x)}{f(x)}
+x\frac{f''(x)}{f(x)}-x\left(\frac{f'(x)}{f(x)}\right)^2\ge 0$$
for all $x\in(0,1)$.
\end{enumerate}
\label{2}
\end{lem}

\begin{proof}
Let $t=\log x$. Then $y=f(x)=f(e^t)$. The convexity of $y$ in $t$ is equivalent to 
$d^2y/dt^2\ge0$. Since
$$\frac{dy}{dt}=f'(e^t)e^t,$$
and
$$\frac{d^2y}{dt^2}=f''(e^t)e^{2t}+f'(e^t)e^t=x(xf''(x)+f'(x)),$$
we obtain the conclusion in part (i).

If $g(x)=f(x^2)$, then it is easy to check that
$$g'(x)+xg''(x)=4x[f'(x^2)+x^2f''(x^2)].$$
So part (ii) follows from part (i).

Similarly, part (iii) follows if we apply part (i) to the function $h(x)=\log f(x)$.
\end{proof}

Recall that $M_{p,\alpha}(f,r)$ is a quotient of two positive functions.
It is thus natural that we will need the following result.

\begin{lem}
Suppose $f=f_1/f_2$ is a quotient of two positive and twice differentiable functions 
on $(0,1)$. Then
\begin{equation}
D(f(x))=D(f_1(x))-D(f_2(x))
\label{eq1}
\end{equation}
for $x\in(0,1)$. Consequently, $\log f(x)$ is convex in $\log x$ if and only if
\begin{equation}
D(f_1(x))-D(f_2(x))\ge 0
\label{eq2}
\end{equation}
on $(0,1)$.
\label{3}
\end{lem}

\begin{proof} 
Observe that
$$D(f(x))=\left(\frac{xf'(x)}{f(x)}\right)'=\left(x(\log f(x))'\right)'.$$
Since $\log f=\log f_1-\log f_2$, we obtain the identity in (\ref{eq1}). By 
part (iii) of Lemma~\ref{2}, $\log f(x)$ is convex in $\log x$ if and only if
inequality (\ref{eq2}) holds.
\end{proof}

To simplify notation, we are going to write 
$$x=r^2,\qquad M(r)=M_p(f,\sqrt{r}).$$
Without loss of generality, we assume throughout the paper that $f$ is not
a constant, so that $M$ and $M'$ are always positive.

We also write
$$h=h(x)=\int_0^r M_p(f,t)(1-t^2)^\alpha\,2tdt=\int_0^xM(t)(1-t)^\alpha\,dt,$$
and
$$\varphi=\varphi(x)=\int_0^r(1-t^2)^\alpha\,2tdt=\int_0^x(1-t)^\alpha\,dt.$$
By part (ii) of Lemma~\ref{2}, the logarithmic convexity of $M_{p,\alpha}(f,r)$ 
on $(0,1)$ is equivalent to the logarithmic convexity of $h(x)/\varphi(x)$ on 
$(0,1)$. According to Lemma~\ref{3}, this will be accomplished if we can show 
that the difference
\begin{equation}
\Delta(x)\equiv D(h(x))-D(\varphi(x))
\label{eq3}
\end{equation}
is nonnegative on $(0,1)$. This will be done in the next section.

W will need several preliminary estimates on the functions $h$ and $\varphi$. The 
next lemma shows where and why we need the assumption $-2\le\alpha\le0$.

\begin{lem}
Suppose $-2\le\alpha\le 0$ and $x\in[0,1)$. Then
\begin{enumerate}
\item[(a)] $1-(\alpha+1)\varphi(x)-(1-x)\varphi'(x)=0$.
\item[(b)] $\varphi(x)-x\ge0$.
\item[(c)] $g_1(x)\equiv x(1-x-\alpha x)-(1-x)\varphi(x)\ge0$.
\item[(d)] $g_2(x)\equiv (\alpha+2)\varphi^2(x)-2(1+x+\alpha x)\varphi(x)+2x\ge0$.
\item[(e)] $g_3(x)\equiv\varphi^2(x)-(1+x+\alpha x)\varphi(x)+x\ge0$.
\end{enumerate}
\label{4}
\end{lem}

\begin{proof} 
If $\alpha\not=-1$, part (a) follows from the facts that
$$\varphi(x)=\frac{1-(1-x)^{\alpha+1}}{\alpha+1},\quad \varphi'(x)=(1-x)^\alpha.$$
If $\alpha=-1$, part (a) follows from the fact that
$$\varphi'(x)=\frac1{1-x}.$$

Part (b) follows from the fact that $(1-t)^\alpha\ge1$ for $\alpha\le0$ and $t\in[0,1)$.

A direct computation shows that
$$g_1'(x)=1-2x-2\alpha x+\varphi(x)-(1-x)\varphi'(x).$$
It follows from part (a) that
$$g_1'(x)=(\alpha+2)(\varphi(x)-x)-\alpha x.$$
By part (b) and the assumption that $-2\le\alpha\le0$, we have $g_1'(x)\ge0$ for
$x\in[0,1)$. Thus $g_1(x)\ge g_1(0)=0$ for all $x\in [0,1)$. This proves (c).

Another computation gives
\begin{eqnarray*}
g_2'(x)&=&2(\alpha+2)\varphi(x)\varphi'(x)\!-\!2(\alpha+1)\varphi(x)\!-\!
2(1+x+\alpha x)\varphi'(x)\!+\!2\\
&=&2(\alpha+2)\varphi(x)\varphi'(x)-2(1+x+\alpha x)\varphi'(x)+2(1-x)^{\alpha+1}\\
&=&2(\alpha+2)(\varphi(x)-x)\varphi'(x).
\end{eqnarray*}
Since 
$$\alpha+2\ge0,\quad \varphi(x)-x\ge0,\quad\varphi'(x)=(1-x)^\alpha\ge0,$$
we have $g_2'(x)\ge0$ for all $x\in[0,1)$. Therefore, $g_2(x)\ge g_2(0)=0$ for 
all $x\in [0,1)$. This proves (d).

A similar computation produces
\begin{eqnarray*}
g_3'(x)&=&2\varphi(x)\varphi'(x)-(\alpha+1)\varphi-(1+x+\alpha x)\varphi'(x)+1\\
&=&2\varphi(x)\varphi'(x)-(1+x+\alpha x)\varphi'(x)+(1-x)^{\alpha+1}\\
&=&(2\varphi(x)-(\alpha+2)x)\varphi'(x)\\
&\ge&2(\varphi(x)-x)\varphi'(x)\ge0,
\end{eqnarray*}
which yields $g_3(x)\ge g_3(0)=0$ for all $x\in [0,1)$. This proves (e) and
completes the proof of the lemma.
\end{proof}

Let us write
\begin{eqnarray*}
&&A=A(x)=\frac{\varphi(x)-x}{\varphi^2(x)},\\
&&B=B(x)=(1-x-\alpha x)+x(1-x)\frac{M'(x)}{M(x)},\\
&&C=C(x)=x(1-x)^{\alpha+1}.
\end{eqnarray*}
By the proof of Lemma~\ref{4}, $A(x)$ is positive on $(0,1)$. $B(x)$ is positive
on $(0,1)$ as $\alpha\le0$ and $M'/M>0$. It is obvious that $C(x)$ is positive 
on $(0,1)$ as well.

\begin{lem}
We have $B^2-4AC>0$ on $(0,1)$.
\label{5}
\end{lem}

\begin{proof}
We have
$$B^2-4AC=\left[(1-x-\alpha x)+x(1-x)\frac{M'}M\right]^2-4x(1-x)^{\alpha+1}
\frac{\varphi-x}{\varphi^2}.$$
It follows from part (a) of Lemma~\ref{4} and the identity 
$\varphi'(x)=(1-x)^\alpha$ that
$$(\alpha+1)x\varphi=x-x(1-x)^{\alpha+1}.$$
Rewrite this as
$$-(1-x-\alpha x)\varphi+\varphi-x=-x(1-x)^{\alpha+1},$$
from which we obtain
$$-4x(1-x)^{\alpha+1}\frac{\varphi-x}{\varphi^2}=-\frac{4(\varphi-x)(1-x-\alpha x)}{\varphi}
+\frac{4(\varphi-x)^2}{\varphi^2}.$$
Combining this with the earlier experession for $B^2-4AC$, we see that $B^2-4AC$ is equal to
the sum of
$$\left[(1-x-\alpha x)-\frac{2(\varphi-x)}{\varphi}\right]^2$$
and
$$x^2(1-x)^2\left(\frac{M'}M\right)^2+2x(1-x)(1-x-\alpha x)\frac{M'}M.$$
The first summand above is always nonnegative, while the second summand is always positive,
because $\alpha\le0$, $M'>0$, and $M>0$. This proves the desired result.
\end{proof}

\section{Proof of Main Result}

This section is devoted to the proof of Theorem~\ref{1}. As was remarked in the
previous section, we just need to show that the difference function $\Delta(x)$
defined in (\ref{eq3}) is always nonnegative on $(0,1)$. Continuing the 
convention in \cite{WZ}, we will also use the notation $A\sim B$ to mean
that $A$ and $B$ have the same sign.

Since
$$\varphi'=(1-x)^\alpha,\quad \varphi''=-\alpha(1-x)^{\alpha-1},$$
we have
\begin{eqnarray*}
D(\varphi(x))&=&\frac{\varphi\varphi'+x\varphi\varphi''-x(\varphi')^2}{\varphi^2}\\
&=&\frac{(1-x)^{\alpha-1}}{\varphi^2}\left[\varphi-x[(\alpha+1)\varphi+
(1-x)^{\alpha+1}]\right].
\end{eqnarray*}
By part (a) of Lemma~\ref{4}, 
$$(\alpha+1)\varphi+(1-x)^{\alpha+1}=(\alpha+1)\varphi+(1-x)\varphi'=1.$$
Therefore,
$$D(\varphi)=\frac{\varphi(x)-x}{\varphi^2(x)}(1-x)^{\alpha-1}.$$
On the other hand,
$$h'=h'(x)=M(x)(1-x)^\alpha,$$
and
$$h''=h''(x)=[(1-x)M'(x)-\alpha M(x)](1-x)^{\alpha-1}.$$
It follows from simple calculations that
$$D(h)=\frac{hh'+xhh''-x(h')^2}{h^2}
=\frac{(1-x)^{\alpha-1}M}{h^2}\left[hB-CM\right].$$
Therefore,
\begin{eqnarray*}
\Delta(x)&=&\frac{(1-x)^{\alpha-1}M}{h^2}(hB-CM)-(1-x)^{\alpha-1}A\\
&\sim&M(hB-CM)-Ah^2\\
&=&-A h^2+MBh-CM^2.
\end{eqnarray*}
The function $\Delta(x)$ is continuous on $[0,1)$, so we just need to show that 
$\Delta(x)\ge0$ for $x$ in the open interval $(0,1)$.

For $x\in(0,1)$ we have $A>0$ and $M>0$, so
\begin{eqnarray*}
\Delta(x)\ge0&\Longleftrightarrow&Ah^2+CM^2\le hBM\\
&\Longleftrightarrow&\frac{h^2}{M^2}+\frac CA\le\frac{hB}{MA}\\
&\Longleftrightarrow&\frac{h^2}{M^2}-\frac{hB}{MA}+\frac{B^2}{4A^2}\le\frac{B^2}{4A^2}
-\frac CA\\
&\Longleftrightarrow&\left(\frac hM-\frac B{2A}\right)^2\le\frac{B^2-4AC}{4A^2}.
\end{eqnarray*}
Recall from Lemma~\ref{5} and the remark preceding it that $A>0$ and $B^2-4AC\ge0$. 
Thus the proof of Theorem~\ref{1} will be completed if we can show that
\begin{equation}
-\frac{\sqrt{B^2-4AC}}{2A}\leq \frac{h}{M}-\frac{B}{2A}\leq\frac{\sqrt{B^2-4AC}}{2A}.
\label{eq4}
\end{equation}

Since the function $M$ is positive and increasing, we have
$$B(x)\ge1-x-\alpha x\ge0,\quad h(x)\leq\int_0^xM(x)(1-t)^\alpha dt=M(x)\varphi(x).$$
It follows from this, the proof of Lemma~\ref{5}, part (b) of Lemma~\ref{4}, and the
triangle inequality that
\begin{eqnarray*}\frac{B+\sqrt{B^2-4AC}}{2A}&\ge&\frac{(1-x-\alpha x)+\left|1-x-
\alpha x-\frac{2(\varphi-x)}{\varphi}\right|}{2A}\\
&\ge&\frac{\frac{2(\varphi-x)}{\varphi}}{2A}=\varphi\ge\frac{h}{M}.
\end{eqnarray*} 
This proves the right half of (\ref{eq4}).

To prove the left half of (\ref{eq4}), we write
$$\delta=\delta(x)=h-M\frac{B-\sqrt{B^2-4AC}}{2A}$$ 
for $x\in(0,1)$ and proceed to show that $\delta(x)$ is always nonnegative.
It follows from the elementary identity
$$\frac{B-\sqrt{B^2-4AC}}{2A}=\frac{2C}{B+\sqrt{B^2-4AC}}$$ 
that $\delta(x)\to0$ as $x\to 0^+$. If we can show that $\delta'(x)\ge0$
for all $x\in(0,1)$, then we will obtain
$$\delta(x)\ge\lim_{t\to0^+}\delta(t)=0,\quad x\in(0,1).$$ 
The rest of the proof is thus devoted to proving the
inequality $\delta'(x)\ge0$ for $x\in(0,1)$.

By direct computation, we have
\begin{eqnarray*}
\delta'(x)&=&M\varphi'-\frac{M'A-MA'}{2A^2}\left[B-\sqrt{B^2-4AC}\right]\\
&-&\frac{M}{2A}\left[B'-\frac{BB'-2(A'C+AC')}{\sqrt{B^2-4AC}}\right].
\end{eqnarray*}
By part (ii) of Lemma~\ref{2} and Hardy's Convexity Theorem, $M$ is 
logarithmically convex. According to part (iii) of Lemma~\ref{2}, the
logarithmic convexity of $M$ is equivalent to
$$\left(x\,\frac{M'}M\right)'=D(M(x))\ge0.$$
It follows that
\begin{eqnarray*}
B'&=&-(\alpha+1)-x\frac{M'}{M}+(1-x)\left(x\frac{M'}{M}\right)'\\
&\ge&-(\alpha+1)-x\frac{M'}{M}\equiv B_0.
\end{eqnarray*} 
Therefore,
\begin{eqnarray*}
\delta'&\ge&M\varphi'-\frac{M'A-MA'}{2A^2}\left[B-\sqrt{B^2-4AC}\right]\\
&&-\ \frac{M}{2A}\left[B_0-\frac{BB_0-2(A'C+AC')}{\sqrt{B^2-4AC}}\right]\\
&\sim&x(1-x)\left[2A^2\varphi'-A\left(\frac{M'}{M}B+B_0\right)+BA'\right]
\sqrt{B^2-4AC}\\
&&+\ x(1-x)\left[AB\left(\frac{M'}{M}B+B_0\right)-A'B^2+2AA'C\right]\\
&&-\ 2x(1-x)A^2\left(2C\frac{M'}{M}+C'\right)\\
&\equiv&d.
\end{eqnarray*}
Here $\sim$ follows from multiplying the expression on its left by the
positive function
$$\frac{2x(1-x)A^2}{M}.$$

We will show that $d\ge0$ for all $x\in(0,1)$. To this end, we are going
to introduce seven auxiliary functions. More specifically, we let
\begin{eqnarray*}
y&=&x(1-x)\frac{M'}M,\\
A_1&=&x(1-x)A'(x)\\
   &=&\frac x{\varphi^3}\left[(\alpha+1)\varphi^2-(2+x+2\alpha x)\varphi+2x\right],\\
B_1&=&x(1-x)\left(\frac{M'}MB+B_0\right)\\
   &=&-(\alpha+1)x(1-x)+(1-2x-\alpha x)y+y^2,\\
C_1&=&x(1-x)\left(2C\frac{M'}M+C'\right)\\
   &=&x(1-(\alpha+1)\varphi)(1-2x-\alpha x+2y),\\
E&=&2A^2C-AB_1+A_1B,\\
F&=&ABB_1-A_1B^2+2AA_1C-2A^2C_1,\\
S&=&\sqrt{B^2-4AC}.
\end{eqnarray*}
Note that the computation for $A_1$ above uses part (a) of Lemma~\ref{4}; the 
computation for $B_1$ uses the definitions of $y$, $B$, and $B_0$; and the 
computation for $C_1$ uses the identities 
\begin{eqnarray*}
C&=&x(1-x)\varphi',\\
C'&=&(1-x)^{\alpha+1}-(\alpha+1)x(1-x)^\alpha\\
&=&(1-x)\varphi'-(\alpha+1)x\varphi',\\
(1-x)\varphi'&=&1-(\alpha+1)\varphi.
\end{eqnarray*}
In terms of these newly introduced functions we can rewrite $d=ES+F$.

It is easy to see that we can write every function appearing in $E$, $F$, and $S$ as 
a function of $(x,y,\varphi)$. In fact, we have
\begin{eqnarray*}
E&=&\frac{x^2}{\varphi^4}(1-(\alpha+1)\varphi)[(\alpha+2)\varphi^2-2(1+x+\alpha
x)\varphi+2x]\\
&+&\frac{1}{\varphi^3}[(3x+2\alpha x-1)\varphi^2-x(1+3x+
3\alpha x\varphi+2x^2]y-\frac{\varphi-x}{\varphi^2}y^2,
\end{eqnarray*}
and
\begin{eqnarray*}
F&=&\frac{x^2}{\varphi^5}(1-(\alpha+1)\varphi)[(\alpha+2)\varphi^2-2(1+x+\alpha
x)\varphi+2x]\\
&&\times\,[(1+x+\alpha x)\varphi-2 x]\\
&+&\frac{1}{\varphi^4}[(1-2x+5x^2-\alpha x+8\alpha x^2+3\alpha^2x^2)\varphi^3\\
&-&x(1+6x+5x^2+5\alpha x+10\alpha x^2+5\alpha^2x^2)\varphi^2\\
&+&4x^2(1+2x+2\alpha x)\varphi-4x^3]y\\
&+&\frac{1}{\varphi^3}[(2-4x-3\alpha x)\varphi^2+4(\alpha+1)x^2\varphi-2x^2]y^2
+\frac{\varphi-x}{\varphi^2}y^3.
\end{eqnarray*}
Note that we have verified the formulas above for $E$ and $F$ with the help of 
Maple. Also, it follows from the proof of Lemma~\ref{5} that
$$S=\sqrt{y^2+2(1-x-\alpha x)y+\frac{1}{\varphi^2}((1+x+\alpha
x)\varphi-2x)^2}.$$
Another tedious calculation with the help of Maple shows that $F^2-E^2S^2$ 
is equal to
$$\frac{4yx^2}{\varphi^8}(\varphi-x)^3(1-(\alpha+1)\varphi)[\varphi^2
-(1+x+\alpha x)\varphi+x](y-y_0),$$
where
$$y_0=\frac{[x(1-x-\alpha x)-(1-x)\varphi][(\alpha+2)\varphi^2-2(1+x
+\alpha x)\varphi+2x]}{(\varphi-x)[\varphi^2-(1+x+\alpha x)\varphi+x]}.$$
This together with Lemma~\ref{4} tells us that
\begin{equation}
F^2-E^2S^2\sim y-y_0.
\label{eq5}
\end{equation}
By Lemma \ref{4} again, we always have $y_0\ge0$.

Recall that $E$, $F$, and $S$ are formally algebraic functions of $(x,y,\varphi)$,
where $x\in(0,1)$, $y\ge0$, and $\varphi>0$. For the remainder of this proof, we 
fix $x$ (hence $\varphi$ as well) and think of $E=E(y)$, $F=F(y)$, and $S=S(y)$ as 
functions of a single variable $y$ on $[0,\infty)$. Thus $E$ is a quadratic 
function of $y$, $F$ is a cubic polynomial of $y$, and $S$ is the square root of a 
quadratic function that is nonnegative for $y\in[0,\infty)$. There are two cases 
for us to consider: $0\le y\le y_0$ and $y>y_0$.

Recall that
$$E(0)=\frac{x^2}{\varphi^4}(1-(\alpha+1)\varphi)[(\alpha+2)\varphi^2-2(1+x+\alpha
x)\varphi+2x].$$
It follows from Lemma~\ref{4} that $E(0)\ge0$. Also, direct 
calculations along with Lemma~\ref{4} show that
\begin{eqnarray*}
E(y_0)&=&\frac{(1-(\alpha+1)\varphi)(\varphi-x)^4[(\alpha+2)\varphi^2-2
(1+x+\alpha x)\varphi+2x]}{\varphi^4[\varphi^2-(1+x+\alpha x)\varphi+x]^2} \\
&\ge&0.
\end{eqnarray*}
Similarly, direct computations along with Lemma~\ref{4} give us
\begin{eqnarray*}
F(y_0)&=&\frac{(1\!-\!(\alpha+1)\varphi)(\varphi-x)^3[(\alpha+2)\varphi^2-2(1+x+
\alpha x)\varphi+2x]}{\varphi^5[\varphi^2\!-\!(1+x+\alpha x)\varphi+x]^3}\\
&&\times\,\left[x[\varphi^2-(1+x+\alpha x)\varphi+x]^2-(1-(\alpha+1)\varphi)
(\varphi-x)^3\right]\\
&\sim&x[\varphi^2-(1+x+\alpha x)\varphi+x]^2-(1-(\alpha+1)\varphi)(\varphi-x)^3.
\end{eqnarray*}
For $x\in(0,1)$ and $\alpha\in[-2,0]$ we have
\begin{eqnarray*}
&[\varphi^2-(1+x+\alpha x)\varphi+x]-(1-(\alpha+1)\varphi)(\varphi-x)\\
&=(\alpha+2)\varphi^2-2(1+x+\alpha x)\varphi+2x>0,
\end{eqnarray*}
and
\begin{eqnarray*}
&x[\varphi^2-(1+x+\alpha x)\varphi+x]-(\varphi-x)^2\\
&=\varphi[x(1-x-\alpha x)-(1-x)\varphi]>0.
\end{eqnarray*} 
It follows that $F(y_0)>0$.

Since $E(y)$ is a quadratic function that is concave downward, it is nonnegative
if and only if $y$ belongs to a certain closed interval. This closed interval
contains $0$ and $y_0$, so it must contain $[0,y_0]$ as well. Therefore, 
$E(y)\ge0$ for $0\le y\le y_0$. It follows from this and (\ref{eq5}) that
$$d\ge E(y)S(y)-|F(y)|\sim E^2(y)S^2(y)-F^2(y)\sim y_0-y\ge0$$
for $0\le y\le y_0$.

In the case when $y>y_0$, we have
$$F^2(y)-E^2(y)S^2(y)\sim y-y_0>0.$$ 
In particular, $F(y)$ is nonvanishing on $(y_0,\infty)$. Since $F(y)$ is
continuous on $[y_0,\infty)$ and $F(y_0)>0$, we conclude that $F(y)>0$ for
all $y>y_0$. Combining this with (\ref{eq5}), we obtain
$$d\ge F(y)-|E(y)|S(y)\sim F^2(y)-E^2(y)S^2(y)\sim y-y_0>0.$$
This shows that $d$ is always nonnegative and completes the proof of Theorem~\ref{1}.

\section{Further results and remarks}

The proof of Theorem~\ref{1} in the previous section actually gives the following 
more general result.

\begin{thm} 
Let $0<p<\infty$ and $-2\le\alpha\le0$. If $M(x)$ is non-decreasing and $\log M(x)$ 
is convex in $\log x$ for $x\in(0,1)$, then the function 
$$x\mapsto \log\frac{\displaystyle\int_0^xM(t)(1-t)^\alpha dt}
{\displaystyle\int_0^x(1-t)^\alpha dt}$$ 
is also convex in $\log x$ for $x\in (0,1)$. 
\label{6}
\end{thm}

The logarithmic convexity of $M_{p,\alpha}(f,r)$ is equivalent to following: 
if $0<r_1<r_2<1$, $0<\theta<1$, and $r=r_1^\theta r_2^{1-\theta}$, then
$$M_{p,\alpha}(f,r)\le \big(M_{p,\alpha}(f,r_1)\big)^\theta 
\big(M_{p,\alpha}(f,r_2)\big)^{1-\theta}.$$
Furthermore, equality occurs if and only if $\log M_{p,\alpha}(f,r)=a\log r+b$ 
for some constants $a$ and $b$, which appears to happen only in very special
situations. For example, if $\alpha=0$, then it appears that $M_{p,0}(f,r)=
ce^{ar}$ (where $c$ and $a$ are constants) only when $f$ is a monomial.

Finally we mention that for $\alpha<-2$ and $y<y_0$, we have
$$\lim_{x\to1}[(\alpha+2)\varphi^2-2(1+x+\alpha x)\varphi+2x]=-\infty.$$
Thus $E(y)<0$ for $x$ close enough to $1$. This implies that $ES+F<0$ for $x$ close 
enough to $1$, so $d$ (and $\delta'$) is not necessarily positive for all $x\in[0,1)$. 
Thus the proof of Theorem \ref{1} breaks down here in the case $\alpha<-2$. However, 
$\delta$ can still be positive. It is just that our approach does not work any more.


\begin{thebibliography}{99}

\bibitem{D} P. Duren, {\it Theory of $H^p$ Spaces}, Academic Press, 1970.

\bibitem{HKZ} H. Hedenmalm, B. Korenblum, and K. Zhu, {\it Theory of Bergman spaces}, 
Springer-Verlag, New York, 2000.

\bibitem{WX} C. Wang and J. Xiao, Gaussian integral means of entire functions, 
preprint, 2013.

\bibitem{WZ} C. Wang and K. Zhu, Logarithmic convexity of area integral means for 
analytic functions, to appear in {\it Math. Scand.}.

\bibitem{XX} J. Xiao and W. Xu, Weighted integral means of mixed areas and lengths 
under holomorphic mappings, {\it arXiv}:1105.6042 [math.CV].
 
\bibitem{XZ} J. Xiao and K. Zhu, Volume integral means of holomorphic functions, 
{\it Proc. Amer. Math. Soc.} {\bf 139} (2011), 1455-1465.

\end{thebibliography}
\end{document}